\documentclass[11pt]{article}

\usepackage{geometry}
\usepackage[english]{babel}
\usepackage{lmodern}
\usepackage{microtype}
\usepackage{changepage}

\usepackage{graphicx}
\usepackage{authblk}
\usepackage{hyperref}
\usepackage[dvipsnames]{xcolor}
\usepackage{standalone}
\usepackage{tikz}
\usepackage{amsmath,amsfonts,amssymb}
\usepackage{amsthm}
\usepackage[capitalise]{cleveref}
\usepackage[shortlabels]{enumitem}
\usepackage{mathabx}
\usepackage{thm-restate}
\usepackage{thmtools}

\graphicspath{{img/}}

\usepackage{tabularx}

\newcommand*\sg[1]{\{ #1 \}}

\newcommand{\Z}{\ensuremath{\mathbb{Z}}}

\newcommand{\F}{\ensuremath{{\mathcal{F}}}}

\newcommand{\Cay}{\mathrm{Cay}}

\renewcommand{\epsilon}{\varepsilon}

\newcommand{\Xlift}{\ensuremath{X^{\nwarrow }}}
\newcommand{\Xrightlift}{\ensuremath{X^{\nearrow}}}

\newtheorem{theorem}{Theorem}
\newtheorem*{theorem*}{Theorem}
\newtheorem{conjecture}{Conjecture}
\newtheorem{lemma}[theorem]{Lemma}
\newtheorem*{lemma*}{Lemma}

\crefname{coro}{Corollary}{Corollaries}

\newtheorem{proposition}[theorem]{Proposition}
\crefname{proposition}{Proposition}{Propositions}

\crefname{prop}{Property}{Properties}

\theoremstyle{definition}
\newtheorem{definition}{Definition}

\theoremstyle{remark}
\newtheorem{remark}{Remark}
\crefname{remark}{Remark}{Remarks}

\crefname{ex}{Example}{Examples}

\title{Period-rigidity of one-relator groups}

\author{Solène J. Esnay\footnote{supported by ANR project IZES (ANR-22-CE40-0011)}}
\affil{\href{mailto:solene.esnay@univ-amu.fr}{solene.esnay@univ-amu.fr}\\
I2M, Aix-Marseille Université, Marseille, France}

\author{Ugo Giocanti\footnote{supported by the National Science Center of Poland
under grant 2022/47/B/ST6/02837 within the OPUS 24 program, and partially supported by the French ANR Project TWIN-WIDTH
(ANR-21-CE48-0014-01).}}
\affil{\href{mailto:ugo.giocanti@uj.edu.pl}{ugo.giocanti@uj.edu.pl}\\{Theoretical Computer Science Department, Faculty of Mathematics and Computer Science, Jagiellonian University, Krak\'ow,
 Poland}\\
}

\author{Etienne Moutot}
\affil{\href{mailto:etienne.moutot@math.cnrs.fr}{etienne.moutot@math.cnrs.fr}\\
CNRS, I2M, Aix-Marseille Université, Marseille, France}

\date{\today}

\begin{document}

\maketitle

\begin{abstract}
We follow in this paper a recent line of work, consisting in characterizing the \emph{periodically rigid} finitely generated groups, i.e., the groups for which every subshift of finite type which is weakly aperiodic is also strongly aperiodic. In particular, we show that every finitely generated group admitting a presentation with one reduced relator and at least $3$ generators is periodically rigid if and only if it is either virtually cyclic or torsion-free virtually $\mathbb Z^2$. This proves a special case of a recent conjecture of Bitar (2024). We moreover prove that period rigidity is preserved under taking subgroups of finite indices. Using a recent theorem of MacManus (2023), we derive from our results that Bitar's conjecture holds in groups whose Cayley graphs are quasi-isometric to planar graphs.      
\end{abstract}

\section{Introduction}
    Symbolic dynamics were  initially introduced by Morse and Hedlund to understand general dynamical systems \cite{MH38}, by studying their \emph{symbolic} version, which are set of (bi-)infinite words under the action of the shift. A central notion in this field is the one of \emph{subshift}: a subshift is defined as the set of all (bi-)infinite words avoiding a given set of forbidden patterns. If this set of forbidden patterns is finite, then we say that its associated subshift is a  \emph{subshift of finite type} (or a SFT for short). 
	All these notions admit natural generalizations to higher dimensions. Namely, one can define multidimensional subshifts as colorings of $\Z^d$ avoiding a given set of forbidden patterns.
	This line of research allowed to establish many connections between symbolic dynamics, complexity theory and computability theory \cite{Aubrun_Sablik_2013,Hochman_2008, Hochman_Meyerovitch_2010, Jeandel_Moutot_Vanier_2020}.
	In particular, many problems that admit well-known solutions in the one-dimensional world become undecidable
	when transposed to higher dimensions. In this context, computability theory turned out to provide relevant tools to understand their complexity.
	One can mention for example problems related to periodicity, which lead to highly non-trivial questions when working in $\Z^2$. One of the first results going in this direction is
	the existence of aperiodic tilesets, discovered by Berger in 1966 \cite{Berger}. In dynamical terms, it means that there exists non-empty SFTs for $\Z^2$ containing only aperiodic configurations, which are objects that cannot exist in $\Z$. 
	In fact, when working in dimension $2$ or more, one can introduce several a priori non-equivalent definitions of periodicity. A configuration is called \emph{weakly periodic} if it is invariant by translation of at least one vector, and \emph{strongly periodic} if it consists of a coloring of a finite portion of the space repeated in every possible direction. In this paper, we are interested in the periodic/aperiodic behaviour of a subshift: we say that a subshift with no strongly periodic configuration  is \emph{weakly aperiodic}, and that a subshift with no weakly periodic configuration is \emph{strongly aperiodic}.
	It is well known that the two-dimensional case plays a special role, as weakly and strongly aperiodic SFTs of $\mathbb Z^2$ turn out to be equivalent notions: any weakly periodic configuration from an SFT can be re-arranged into a strongly periodic one, belonging to the same SFT.
	However, this property does not hold anymore in dimension 3 and higher. 
    Note that all the aforementioned definitions generalize in a natural way in finitely generated groups (see Section \ref{sec:def} for more precise definitions).

	As illustrated by the cases of $\Z$, $\Z^2$ and $\Z^3$ above, the periodic behaviour of the SFTs can be very different according to the group one considers.
	A natural question is whether we can have a better understanding of the relationship between the geometrical and structural properties of a group and the properties of its aperiodic SFTs.
	Several conjectures have been formulated aiming to answer to this question.
	
	\begin{conjecture}
	\label{conj: strongly}
		A finitely generated group $G$ admits a strongly aperiodic SFT if and only if $G$ is one-ended and has decidable word problem.
	\end{conjecture}

	Although it is hard to trace the first explicit mention of it in the litterature, \cref{conj: strongly} naturally arises in the light of two known results: if $G$ admits a strongly aperiodic SFT and is recursively presented, then it has a decidable word problem \cite{jeandel2015aperiodicsubshiftsfinitetype}; and on the other hand, any group with at least two ends does not admit a strongly aperiodic SFT \cite{Cohen}.
	Though the general case remains widely open, \cref{conj: strongly} has been shown to be true in several classes of groups, including in particular: virtually polycyclic groups \cite{jeandel2016aperiodicsubshiftspolycyclicgroups}, solvable Baumslag Solitar groups \cite{ME22}, hyperbolic groups \cite{Cohen-GS-Rieck}, lamplighter groups \cite{bartholdi2024shiftslamplightergroup}.
	In the same vein, Carroll and Penland conjectured that every group which is not virtually cyclic should admit a weakly aperiodic SFT.
	\begin{conjecture}[Carroll, Penland \cite{CP15}]
	\label{conj: CP}
		A finitely generated group $G$  admits a weakly aperiodic SFT if and only if it is not virtually cyclic.
	\end{conjecture}
	Here again, \cref{conj: CP} has been proved in many classes of groups, including all the aforementioned ones for which \cref{conj: strongly} holds, together with the class of non-amenable groups \cite{JeandelTranslationlike}.
	
	A similar conjecture has been stated for the decidability of the domino problem, asking if a given set of Wang tiles tiles the group or not, and is know to be true roughly for the same classes of groups.
	\begin{conjecture}[Ballier, Stein \cite{Ballier_Stein_2013}]
	\label{conj: DP}
		A finitely generated group $G$ has undecidable domino problem if and only if it is not virtually free.
	\end{conjecture}

	We refer to Bitar's PhD thesis \cite{BitarPhD} for a more complete list of the state of the art on these conjectures.

	In a recent work \cite{Bitar24}, Bitar initiated a systematic study of the \emph{rigidity} of the periodic configurations, aiming at characterizing in a similar way groups which are \emph{periodically rigid}, i.e., in which every weakly aperiodic SFT is also strongly aperiodic. Such groups are called \emph{periodically rigid}. An example of such a group is $\mathbb Z^2$.
	Interestingly, the existence of weakly or strongly aperiodic SFT for a group is a commensurability invariant, but not its period rigidity (see \cref{rem:lift-not-strongly} below for examples of groups admitting an SFT which is weakly but not strongly aperiodic, commensurable to some periodically rigid groups).
	
	In the vein of the previous two conjectures about periodicity, Nicolas Bit\'ar conjectured the following:
	\begin{conjecture}[Bitar \cite{Bitar24}]
	 \label{conj: Bitar}
	 Let $G$ be a finitely generated group. Then $G$ is periodically-rigid if and only if it is either virtually cyclic or torsion-free virtually $\mathbb Z^2$.
	\end{conjecture}
	\noindent
	In the same paper, the author provides several constructions allowing to transfer period rigidity of a group to some other related groups, and proves that Conjecture \ref{conj: Bitar} holds for virtually nilpotent and polycyclic groups.

	In this paper, we show that Conjecture \ref{conj: Bitar}  also holds for groups admitting a presentation with one relator and at least three generators.
	
    \begin{theorem}
	 \label{thm: one-relator} 
	 Let $G$ be a group that admits a presentation $\langle S\mid r \rangle$ such that $|S|\geq 3$ and $r$ is a cyclically reduced non-empty word. Then $G$ is not periodically rigid.
	\end{theorem}	
	
	The class of one relator groups 
	has been introduced and studied in an early work of Magnus \cite{Freiheitssatz}, and takes a central place in geometric group theory.
    To our knowledge, Conjecture \ref{conj: Bitar} remains open for groups admitting a finite presentation with one relator and two generators, which includes in particular the class of Baumslag-Solitar groups, for which the question has been solved \cite{ME22}.
    
    A consequence of Theorem \ref{thm: one-relator} is that 
    Conjecture \ref{conj: Bitar} holds for surface groups (i.e, fundamental groups of closed orientable surfaces), namely every surface group of genus at least $2$ admits a weakly but not strongly aperiodic SFT. According to \cite{Cohen-GS-Rieck},  
    Gromov~\cite[Theorem 3.4.C]{Gromov} provided a construction of SFTs which are weakly but not strongly aperiodic in any finitely generated hyperbolic group, expanded upon by Coornaert and Papadopoulos~\cite[Chapter 3, Theorem 8.1]{Coornaert-Papadopoulos}, although without explicitly mentioning the notion of aperiodicity. Our construction thus gives another example of such SFTs in surface groups.
    
    Recently, MacManus \cite[Corollary D]{Mac23} proved a structure theorem for \emph{quasi-planar groups}, that is, finitely generated groups having one (and thus all) of their Cayley graph being quasi-isometric to some planar graph. From a geometric perspective, this class of groups is of special interest 
    as it
    generalizes the class of \emph{planar 
    groups} which has been widely studied over the last two centuries (see for example \cite{Droms, GH22, Maschke, Zieschang80}), and more generally the class of the groups that admit some Cayley graph excluding some  countable graph as a minor, which has been specifically investigated more recently \cite{EG24}. Using Theorem \ref{thm: one-relator} together with MacManus' result, we prove that Conjecture \ref{conj: Bitar} also holds for quasi-planar groups. 
    
	\begin{restatable}{coro}{QP}
    \label{cor: QP}
    Let $G$ be a quasi-planar group. Then $G$ is periodically rigid if and only if it is either virtually cyclic or torsion-free virtually $\mathbb Z^2$.
    \end{restatable}
	
    In particular, our proof of Corollary \ref{cor: QP} relies on a construction of a SFT that we introduce, which implies that period rigidity is preserved when taking finite index subgroups.

    \begin{proposition}[see Lemma \ref{lemma:heredity} below]
    \label{prop: finite-index}
        Let $G,H$ be two finitely generated groups such that $H$ is a finite index subgroup of $G$. If $G$ is periodically rigid, then so is $H$. 
    \end{proposition}
    
    The paper is organized as follow. In \cref{sec:def}, we provide useful definitions and results. 
    \Cref{sec:subgroup} contains a proof of Proposition \ref{prop: finite-index}, and Section \ref{sec:one-relator} contains a proof of Theorem \ref{thm: one-relator} and Corollary \ref{cor: QP}.

\section{Preliminaries}
\label{sec:def}
\subsection{Group presentations, Cayley graphs and subshifts of finite type}

A \emph{presentation} of a group $G$ is a pair $\langle S \mid R \rangle$, where $S$ is a generating subset of $G$, and $R$ is a set of finite words using symbols from $S \sqcup S^{-1}$ called \emph{relators}, such that every element of $G$ can be written as a word of $S \cup S^{-1}$, and such that if two words correspond to the same group element, then one can obtain the second from the first after performing a sequence of deletions and additions of words $ss^{-1}$, $s^{-1}s$ for $s\in S$, and of words $r \in R$ as factors.
In other words, if $\F(S)$ denotes the free group with generating set $S$, then $G$ is isomorphic to $\F(S)/N(R)$, where $N(R)$ denotes the normal subgroup of $\F(S)$ generated by elements of the form $wrw^{-1}$, for $w\in (S\sqcup S^{-1})^{\ast}$ and $r\in R$. 
If $\langle S \mid R \rangle$ is a presentation of $G$, then we will write for short $G=\langle S \mid R \rangle$.

For every group $G$, every set $S$ of generators of $G$ and every $g\in G$, we let $|g|_S\in \mathbb N$ denote the minimum size of a word $w\in (S\cup S^{-1})^{\ast}$ corresponding to $g$. 

A \emph{Tietze transformation} is one of the following operations, which  allows to transform a presentation $\langle S \mid R \rangle$ of a group $G$ into another presentation of $G$.
\begin{itemize}
 \item \emph{Adding a relation to $R$}: we add a relator $r\in (S\uplus S^{-1})^\ast$ that is already representing the identity element of $G$;
 \item \emph{Removing a relation to $R$}: we remove a redundant relation $r$ from $R$, i.e., a relation $r\in R$ such that $r$ also corresponds to the identity element of the group described by the presentation $\langle S\mid R\setminus\sg{r}\rangle$;
 \item \emph{Adding a generator to $S$}: we add a new letter $s$ to $S$ disjoint from $S\uplus S^{-1}$ together with a relation $sw$ for some word $w\in (S\uplus S^{-1})^\ast$;
 \item \emph{Removing a generator from $S$}: if some relation $r$ can be written $r=sr'$, where neither $s$ nor $s^{-1}$ occur in $r'$, we remove $r$ from $R$ and $s$ from $S$, and replace every occurence of $s$ (respectively of $s^{-1}$) in the other relations by $(r')^{-1}$ (respectively by $r'$).
\end{itemize}

Given two groups $G=\langle S \mid R \rangle$ and $H=\langle T \mid U \rangle$ such that $S\cup S^{-1}$ and $T\cup T^{-1}$ are disjoint, their \emph{free product} is the group $G * H := \langle S, T \mid R,U \rangle$. A group $G$ is \emph{virtually} $\mathcal P$, for a given property $\mathcal{P}$ if it has a subgroup of finite index satisfying $\mathcal{P}$.

A group $G$ is said to have \emph{torsion} if it admits at least one nontrivial element with finite order.

The \emph{surface group $G_g$ of genus $g$} for $g\geq 0$ is the fundamental group of the closed orientable surface of genus $g$. For each $g\geq 1$, the surface group of genus $g$ can alternatively be defined by the finite presentation 

\[\langle a_1, b_1, \ldots, a_g, b_g \mid [a_1, b_1] \ldots [a_g,b_g] \rangle\]

with $[a_i,b_i] = a_i b_i {a_i}^{-1} {b_i}^{-1}$ the commutator of two elements.

\medskip

Given a finite generating set $S$ of a finitely generated group $G$, the associated (right) \emph{Cayley graph} of $G$ is the (non-oriented) graph $\Cay(G,S)$ whose vertices are the elements of $G$ and with one edge $\sg{g, gs}$, for every $g\in G$ and $s\in S\cup S^{-1}$.

We equip every graph with its shortest path metric. Given two metric spaces $(X,d_X)$ and $(Y,d_Y)$, we say that $X$ is
\emph{quasi-isometric} to $Y$ if there is a map $f: X \rightarrow Y$
and constants $\varepsilon\ge 0$, $\lambda\ge 1$, and $C\ge 0$ such that
(i) for any $y\in Y$ there is $x\in X$ such that $d_Y(y,f(x))\le C$,
and (ii) for every $x_1,x_2\in X$, $$\frac1{\lambda}d_X(x_1,x_2)-\varepsilon\le d_Y(f(x_1),f(x_2))\le \lambda d_X(x_1,x_2)+\varepsilon.$$
It is well known that any two Cayley graphs of a same finitely generated group are quasi-isometric to each other.

A \emph{ray} in a graph $G$ is an infinite simple one-way path $P=(v_1,v_2,\ldots)$. A \emph{subray} $P'$ of $P$ is a ray of the form $P'=(v_i,v_{i+1},\ldots)$ for some $i\geq 1$. We say that a ray \emph{lives} in a set $X\subseteq V(G)$ if one of its subrays is included in $X$.
We define an equivalence relation $\sim$ over the set of rays $\mathcal R(G)$ by letting $P\sim P'$ if and only if for every finite set of vertices $S\subseteq V(G)$, $P$ and $P'$ are living in the same component of $G- S$. The \emph{ends} of $G$ are the elements of classes $\mathcal R(G)/\sim$.
It is not hard to check that the number of ends of a graph is invariant under taking quasi-isometries. In particular, we will talk about the number of ends of a group $\Gamma$, as it does not depend on the choice of its Cayley graph. Freudenthal \cite{Freudenthal44} proved that for every finitely generated group, its number of ends lies in $\sg{0,1,2,\infty}$. 
Note that the $0$-ended groups correspond to the finite groups, while it is well-known (see for example \cite[Proposition 9.23]{Drutu}) that a group has $2$ ends if and only if it is virtually-$\mathbb Z$. 

\medskip

Let $G$ be a finitely generated group with neutral element $1_G$. Let $A$ be a finite set called an \emph{alphabet}.

We say that $A^G$ is the \emph{full shift} with alphabet $A$ on $G$. For a given \emph{configuration} $x \in A^G$, we may use equivalently the notations $x(g)$ and $x_g$ for $g \in G$.
The full shift can be endowed with the natural left action of $G$, called the \emph{shift}: for $g,h \in G$ and $x \in A^G$, we set $(g \cdot x)_h := x_{g^{-1}h}$.
When additionally endowed with the prodiscrete topology (product topology of the discrete topology on $A$), $A^G$ forms a dynamical system.

A \emph{pattern} $p$ is an element of some $A^{P_p}$ where $P_p \subseteq G$ is a finite subset of $G$ called the \emph{support} of $p$.
We say that a pattern $p \in A^{P_p}$ \emph{appears} in a configuration $x\in A^{G}$ -- or that $x$ \emph{contains} $p$ -- if there exists $g \in G$ such that for every $h \in P_p$, $(g \cdot x)_{h} = p_{h}$. In this case, we write $p\sqsubset x$.

The \emph{subshift} associated to a set of patterns $\F$, called set of \emph{forbidden patterns}, is defined by
\[
X_\F = \{ x \in A^{G} \mid \forall p \in \mathcal{F}, p \not\sqsubset x \}
\]
that is, $X_\F$ is the set of all configurations that do not contain any pattern from~$\F$. Note that there can be several sets of forbidden patterns defining the same subshift $X$.
A subshift can equivalently be defined as a closed set under both the topology and the $G$-action. It inherits this topology and this action and can be seen as a dynamical system.
If $X=X_\F$ with $\F$ finite, then $X$ is called a \emph{Subshift of Finite Type}, SFT for short.

Note that we use the left $G$-action in all this paper, because we represent groups using a right Cayley graph, meaning that subshifts can be seen as sets of colorings of a given Cayley graph, and the left action is the one that preserves patterns under a shift.

\subsection{Aperiodicity of subshifts of finite type}
	
	Let $X$ be a subshift on a group $G$ and $x \in X$.
    The \emph{orbit} of $x$ is ${Orb_G(x) = \{g\cdot x \mid g\in G\}}$ and its \emph{stabilizer} $Stab_G(x) = \{ g\in G \mid g\cdot x = x \}$. We say that $x$ is a \emph{strongly periodic configuration} if $|Orb_G(x)| < +\infty$, and that $x$ is a \emph{weakly periodic configuration} if $Stab_G(x) \neq \{e\}$.
    
    If no configuration in $X$ is strongly periodic and the subshift is nonempty, we say that the subshift is \emph{weakly aperiodic}.
    If no configuration in $X$ is weakly periodic and the subshift is nonempty, we say that the subshift is \emph{strongly aperiodic}.
    To avoid confusion, we will always use aperiodicity for subshifts and periodicity for configurations in this paper.
    
    We say that a group is \emph{periodically rigid}, a notion introduced in \cite{Bitar24} if it admits no SFT $X$ that is weakly but not strongly aperiodic. Said otherwise, the two notions of aperiodicity above are equivalent in that group.
	
	We end this section with a few known results about aperiodicity in SFTs.
	The first of them is the main result of the PhD of Piantadosi, which states that nontrivial free groups are not periodically rigid. In this paper, several proofs of weak aperiodicity on other group structures will be obtained by showing they contain a free subgroup, and inheriting the weak aperiodicity of a SFT on the free group through a construction called a free extension (see \cref{def:lift}).

	\begin{theorem}[\cite{Piantadosi}, Theorems 2.2 and 3.5]
	\label{th:piantadosi}
	Any free group of rank at least 2 is not periodically rigid.
	\end{theorem}

	The second is a result of Cohen and Goodman-Strauss who proved that surface groups admit a strongly aperiodic SFT by encoding orbit graphs of substitutions on their Cayley graph. 

	\begin{theorem}[\cite{CGS}, Theorem 2.1]
	\label{th:cgs}
	Surface groups admit a strongly aperiodic SFT.
	\end{theorem}

	Finally, the third known result we will need is due to Cohen (\cite{Cohen}, patched by Salo and Genevois \cite{Genevois_Salo_2019}, which will allow us in what follows to focus on one-ended groups.
 
	\begin{theorem}[\cite{Cohen}, Theorem 1.5]
	\label{th:cohen}
	A group with infinitely many ends cannot have a strongly aperiodic SFT.
	\end{theorem}

\section{Heredity of aperiodicity of subgroups}
\label{sec:subgroup}

    Given a finitely generated group $H$ together with an SFT $X$ of $H$, a natural question is the following: if $H$ is a subgroup of a finitely generated group $G$, how can we extend $X$ to an SFT of $G$ sharing the same properties of aperiodicity?
    In this section, we present two different constructions of SFTs that can be considered as such extensions. The first one is called the \emph{free extension} of $G$ and is well-known and natural to define; however it does not preserves strong aperiodicity. Intuitively, it is based on forcing the rules of $X$ on each \emph{left} $H$-coset.
    
    We call the second construction the \emph{right extension} of $G$. In general, the right extension is only a subshift, which not necessarily of finite type, and it only produces an SFT when $H$ is a subgroup of $G$ finite index. We will restrict to that case, as we focus on the aperiodicity of SFTs. We will show that this extension then preserves both weak and strong aperiodicity. In some way, it can be thought as the best possible attempt to propagate the rules of the SFT $X$ on $H$ to an SFT on $G$, using this time the \emph{right} $H$-cosets.

\subsection{Free extensions}
\label{sec: lift}
	
	\begin{definition}[Free extension]
		\label{def:lift}
		Let $G,H$ be finitely generated groups such that $H$ is a subgroup of $G$. Let $X$ be a subshift on $H$ with forbidden patterns $\F$ and alphabet $A$. As $H\subseteq G$, note that $\F$ also defines a set of forbidden patterns on $G$.
		
		The \emph{free extension} (or \emph{lift}) of $X$ in $G$ is the subshift $\Xlift$ on $G$ defined by the set of forbidden patterns $\F$, when considered as patterns in $G$.
        Said otherwise, we have
		\[
		\Xlift = \{ y \in A^G \mid \forall g \in G, y|_{gH} \in X\}.
		\]
	\end{definition}
    
    The previous equality is due to the fact that a configuration $y \in A^G$ is not in $\Xlift$ if and only if one can find $g \in G$ and some finite subset $P$ of $H$ so that $y|_{gP}$ is a forbidden pattern in $\F$ with support $P$. 
	
	Observe that if $X$ is an SFT, then $\Xlift$ is also an SFT, and that $\Xlift$ forces every left $H$-coset of $G$ to reproduce a configuration of $X$ (all of them being possibly distinct).
    It is not hard to check that free extensions preserve weak aperiodicity.
	
	\begin{proposition}[\cite{JeandelTranslationlike}, Proposition 1.1]
		\label{th:jeandel}
        Let $G,H$ be two finitely generated groups such that $H$ is a subgroup of $G$, and let $X$ be an SFT on $H$.
        If $X$ is weakly aperiodic, then $\Xlift$ is also weakly aperiodic.
	\end{proposition}
	
	\begin{remark}
    \label{rem:lift-not-strongly}
    \cref{th:jeandel} does not hold for strong aperiodicity: using its notations, there might exist some strongly aperiodic SFT $X$ of $H$ such that $\Xlift$ is not a strongly aperiodic SFT of $G$. 
    For example, if $H$ is any finitely generated group that admits a strongly aperiodic SFT $X\subseteq A^H$ (take for example $H:=\mathbb Z^2$ \cite{Berger}), and if $G:=H\times (\mathbb Z/k\mathbb Z)$ for some $k\geq 2$, then it is not hard to see that $\Xlift\subseteq A^G$ is not strongly aperiodic anymore. Let $x\in X$ and define the configuration $x^{\uparrow}: G\to A$ by setting for each $(g, i)\in G$, $x^{\uparrow}(g,i):=x(g)$. Then by definition of $\Xlift$, we have $x^{\uparrow}\in \Xlift$, but for any $i\in (\mathbb Z/k\mathbb Z)\setminus\sg{0}$, the element $g:=(1_H, i)\in G\setminus \sg{1_G}$ is such that $g\cdot x^{\uparrow}=x^{\uparrow}$, implying that $\Xlift$ is not strongly aperiodic.
    In fact this example generalizes to any case where $H$ is a torsion-free subgroup of a group $G$ with torsion \cite[Proposition 6.9]{Bitar24}.
\end{remark}
	
	In \cite{Barbieri}, the author gives precise conditions on $G,H$ and $X$ for $\Xlift$ to be strongly aperiodic on $G$. In particular, the following result that depends purely on the structure of $G$ and $H$ will be useful.
	
	\begin{theorem}[Corollary of \cite{Barbieri}, Proposition 3]
		\label{th:barbieri}
		Let $H$ be a subgroup of $G$. Let $X$ be an SFT on $H$.
		
		If			
		\[
		\exists g \in G \setminus \{1\}, \forall \gamma \in G, \forall n > 0, \gamma g^n \gamma^{-1} \notin H \setminus \{1\},
		\]
		
		then $\Xlift$ is not strongly aperiodic on $G$.
	\end{theorem}
	
\subsection{Right extension}
\label{sec: finite-index}
	\begin{definition}[Right extension]
		\label{def:rightlift}
		Let $H$ be a finite index subgroup of $G$, with a finite set of generators $S_H$.
		Let $X$ be a subshift on $H$, and $\F_H$ denote its associated set of forbidden patterns.
		Let $g_1,\ldots, g_k\in G$ be representatives of the right cosets $H\cdot g_1, \ldots, H\cdot g_k$ of $H$ with $k:=[G:H]\in \mathbb N$ and $B:=A\times [k]$, with $[k]$ abridged notation of $\{1,\dots,k\}$. We assume that $g_1=1_G$.
		  
		For each support $P$ of a pattern $p \in \F_H$ and $i\in [k]$, we set $P_i:={g_i}^{-1}  P   g_i\in G$. For each generator $a \in S_H$ and $i\in [k]$, we also define $a_i:={g_i}^{-1}  a   g_i\in G$ and consider the finite set $S:=\sg{a_i \mid a\in S_H, i\in [k]}$. 
		We define a set $\F$ of forbidden patterns in $G$ with supports $P_i$ and $\sg{1_G, a_i}$ for each $a_i\in S$, using alphabet $B$, as follows:
		\begin{enumerate}
			\item \label{forbid-monochrome} for each $a\in S_H, c_1,c_2\in A, i\in [k]$ and $j\in [k]\setminus \sg{i}$ we add in $\F$ the forbidden pattern $q$ with support $\sg{1_G,a_i}$ defined by $p'(1_G):=(c_1,i)$ and $p'(a_i):=(c_2,j)$;
			\item \label{forbid-X} for each pattern $p\in \F_H$ with support $P$ and each $i\in [k]$, we add in $\F$ the pattern $p'$ with support $P_i$ defined by $p'({g_i}^{-1}qg_i):=(p(q), i)$ for each $q \in P$.
		\end{enumerate}

		We let $\Xrightlift:=X_{\F}\subseteq B^G$ denote the associated subshift, and call it the \emph{right extension} of $X$.
	\end{definition}

	The main idea behind \cref{def:rightlift} is to mirror the definition of $\Xlift$ with respect to right cosets, and to find a way to transmit constraints of $\F_H$ along each coset $H \cdot g_i$. A natural way to do this is to transmit the constraints of $\F_H$ along directions $g_i^{-1}hg_i$, with $h\in H$. This is formally done thanks to the constraints from \cref{forbid-X}.

	\begin{remark}
		The definition of right extension is not to be confused with the higher block shift defined in~\cite{CP15}. That construction is built the other way around: it starts from a subshift $Y \subseteq A^G$ on $G$ and ``collapses'' it into an SFT $X \subseteq (A^{\{g_1, \dots, g_k\}})^H$ on $H$ following right coset representatives; that is, each configuration $x$ stores in position $h$ all of $y(hg_1), \dots, y(hg_k)$.
	\end{remark}
	
    \begin{remark}
     \label{rem: pi2}
     We reuse here the notations from Definition \ref{def:rightlift} and in the remainder of this section, we will let $\pi_1: B\to A$ and $\pi_2: B\to [k]$ denote the projections on the first and second coordinates of $B$, i.e. $(c,i)=(\pi_1(c,i),\pi_2(c,i))$ for each $(c,i)\in B$.
     For any configuration $y \in \Xrightlift$, we say that $g\in G$ is \emph{$i$-colored} by $y$ if $\pi_2(y)=i$. 
     Observe that if for some configuration $y \in \Xrightlift$, some element $g\in H\cdot g_i$ is $i$-colored, then every element from $H\cdot g_i$ must also be $i$-colored in $y$.
     Indeed, if there exist $h \in H$ and $i \in [k]$ so that $\pi_2(y(h g_i))=i$, then the forbidden patterns (\ref{forbid-monochrome}) impose that for each $a\in S_H$, $\pi_2(y(hg_i a_i)) = i$. Moreover, for all $h' \in H, h' g_i a_i = (h'a) g_i$ by definition of $a_i$. So as $S_H$ is a generating set of $H$, we can prove by induction on $|h^{-1}h'|_{S_H}$ that $\pi_2(y(h' g_i))=i$ for all $h'\in H$.
    \end{remark}

	\begin{lemma}
	\label{lemma:heredity}
     Let $G$ be a finitely generated group and $H$ be a finitely generated finite index subgroup of $G$. If $X$ is a weakly but not a strongly aperiodic SFT on $H$, then $\Xrightlift$ is weakly but not strongly aperiodic on $G$. 
     In particular, if $G$ is periodically rigid, then so is $H$.
	\end{lemma}

	\begin{proof}
		With the notations from Definition \ref{def:rightlift}, let $X\subseteq A^H$ be a weakly but not strongly aperiodic SFT of $H$ on some finite alphabet $A$, with set of forbidden patterns $\F$.
		We let $Y:=\Xrightlift\subseteq B^G$ denote the right extension of $X$ and will show that $Y$ is weakly but not strongly aperiodic.

	    \paragraph*{$Y$ is weakly aperiodic}
		We let $y\in Y$ and show that the $G$-orbit of $y$ is infinite. Intuitively, \cref{rem: pi2} implies that, up to a shift, there is a coset $H \cdot g_i$ which is $i$-colored by $y$, and in which $y$ ``simulates'' $X$. In particular, this will imply that $y$ has an infinite $G$-orbit.
		
		Let $i\in [k]$ be such that $y(1_G)=(c,i)$ for some $c\in A$.
	    Then, the translated configuration $z:=g_i\cdot y\in Y$ is such that $z(g_i)=(c, i)$.
	    As $z$ is in the $G$-orbit of $y$, it is enough to prove that its $G$-orbit is infinite so without loss of generality we may assume that $y(g_i)=(c,i)$ for some $c\in A$. In particular, Remark \ref{rem: pi2} implies that $\pi_2(y(h g_i))=i$ for all $h\in H$. Hence, for every configuration $y'\in Y$ belonging to the $H$-orbit of $y$, we also have
	    $\pi_2(y'(g_i))=i$.

	    For every configuration $y'\in Y$ that belongs to the $H$-orbit of $y$, we define a configuration $x^{(y')}\in A^H$ by setting $x^{(y')}(h):=\pi_1(y'(hg_i))$ for each $h\in H$. 
	    
	    We first show that $x^{(y')}\in X$ for any $y'\in Y$ in the $H$-orbit of $y$. Let $p$ be a forbidden pattern of $\F$ with support $P \subseteq H$. Note that the existence of some $h \in H$ so that $\forall q \in P$, $x^{(y')}(hq) = p(q)$, would imply that $y'(hq g_i) = (p(q), i)$ for each $q\in P$, implying that  $y'$ contains one of the forbidden patterns (\ref{forbid-X}), and thus contradicting the fact that $y'\in Y$. Consequently, $x^{(y')}\in X$.

	    Note that for every $h,h'\in H$, 
	    \[ x^{(h\cdot y)}(h') = \pi_1(h\cdot y(h'g_i)) = \pi_1(y(h^{-1} h' g_i)) = x^{(y)}(h^{-1}h') = (h\cdot x^{(y)})(h') , \]
	    thus $x^{(h\cdot y)}=h\cdot x^{(y)} \in X$ for each $h\in H$. As $X$ is weakly aperiodic, $x^{(y)}$ has an infinite $H$-orbit, and the previous equality thus implies that $y$ also has an infinite $H$-orbit (since $x^{(y)} \neq x^{(y')}$ implies $y \neq y'$). In particular, $y$ has an infinite $G$-orbit.

	    \paragraph*{$Y$ is nonempty and not strongly aperiodic}
	    As $X$ is not strongly aperiodic, there exist $x\in X$ and $h_0\in H\setminus \sg{1_H}$ such that $h_0\cdot x=x$. 
	    We define $y\in B^G$ by setting for every $i\in [k]$ and $h\in H$, $y(h g_i):=(x(h),i)$.

	    First, we show that $y\in Y$. 
	    By definition of $y$, if $y(g) = (c_1,i)$ for some $(c_1, i)\in B$, then $g=hg_i$ for some $h \in H$. Consequently, $y(ga_i) = y(hg_i{g_i}^{-1}ag_i) = y(hag_i) = (c_2,i)$, where $c_2:=x(ha)\in A$, once again by definition of $y$. This ensures that none of the forbidden patterns (\ref{forbid-monochrome}) appear in $y$.
	    
	    Now, let $p \in \F$ with support $P \subseteq H$. If there is some $g$ so that $y(g{g_i}^-1qg_i) = (p(q), i)$ for all $q \in P$, then that means $g{g_i}^-1qg_i \in Hg_i$, and notably $g{g_i}^-1 \in H$. This implies that we can write $g=hg_i$ for some $h \in H$. As a consequence, $y(g{g_i}^-1qg_i) = y(hqg_i) = (x(hq), i)$ for all $q \in P$. This means that for all $q\in P$, $x(hq) = p(q)$, which contradicts the fact that $x \in X$. Hence, none of the forbidden patterns (\ref{forbid-X}) appear in $y$.
		As a consequence, $y \in Y$.
    
        \smallskip

	    Finally, let us show that $h_0$ is a non-trivial period of $y$. For any $h \in H, i \in [k]$, we have
        \[ h_0 \cdot y(h g_i) = y(h_0^{-1} h g_i) = (x(h_0^{-1} h), i) = (x(h), i) = y(h g_i),\]
        implying that $h_0\cdot y=y$ since all elements of $G$ can be written uniquely as a $h g_i$. We thus conclude that $X$ is not strongly aperiodic, as desired.
	\end{proof}

\section{One-relator groups with at least 3 generators}
\label{sec:one-relator}

In this section, we give a proof of \cref{thm: one-relator}.
Our proof will use the `Freiheitssatz', a seminal theorem of Magnus, which plays a central role in the theory of one-relator groups.

\begin{theorem}[Freiheitssatz \cite{Freiheitssatz}]
\label{th:freiheitssatz}
Let $G = \langle S | r \rangle$ be a one-relator group with $r$ cyclically reduced, such that $s \in S$ appears in $r$. Then the subgroup of $G$ generated by $S \setminus{\{s\}}$ is free of rank $|S|-1$.
\end{theorem}
	
In the remainder of the section, if $G$ is a group with a set of generators $S$, for every $c\in S\cup S^{-1}$ and every word $w\in (S\cup S^{-1})^{\ast}$, we denote with $\#_cw \in \mathbb N$ the number of occurences of $c$ in $w$. We moreover set $|w|_c:= \#_cw-\#_{c^{-1}}w$, so that $|w|_c\in \Z$.
	
	Our proof of \cref{thm: one-relator} is based on the following three lemmas:

	\begin{lemma}
		\label{lemma:zero-aperiodic}
	  Let $G = \langle S | r \rangle$ with $|S|\geq 3$, $S=\{a,b,c,\hdots\}$, and $r$ cyclically reduced. Suppose that $c$ appears in $r$ and $|r|_c = 0$. Then $G$ is not periodically rigid.
	\end{lemma}

	\begin{proof}
		By \cref{th:freiheitssatz}, the subgroup $H:=\langle a,b \rangle$ of $G$ is a free group of rank $2$. Thus by \cref{th:piantadosi}, $H$ admits a weakly aperiodic SFT $X$. By \cref{th:jeandel}, $\Xlift$ is therefore a weakly aperiodic SFT on $G$.
		
		Now, our aim is to prove that $\Xlift$ is not strongly aperiodic using \cref{th:barbieri}. Note that for every $g\in G$, and every word $w$ in $(S \cup S^{-1})^*$ that represents $g$, the value $|w|_c$ does not depend on the choice of the word $w$ representing $g$. Indeed,
		if $w,w'$ are two words representing the element $g$, then there is a (finite) sequence of operations allowing to transform $w$ into $w'$, where at each step, 
		we either remove or add some factor of the form $ss^{-1}$ or $s^{-1}s$ for some $s\in S$, or some factor which is a copy of the relator $r$.
        Note that none of these operations change the value $|w|_c$. Alternatively, this means that $|.|_c: G\to \Z$ defines a group homomorphism. 
		
		Now observe that for every $\gamma\in G, n>0$, we have $|\gamma c^n \gamma^{-1}|_c = |c^n|_c = n > 0$, thus $\gamma c^n \gamma^{-1}\notin\langle a,b \rangle$. As a consequence, using~\cref{th:barbieri}, we conclude that $\Xlift$ is not strongly aperiodic.
	\end{proof}

	\begin{lemma}
		\label{lemma:infiniteends-aperiodic}
		Let $G$ be a finitely generated group with infinitely many ends. Then $G$ is not periodically rigid.
	\end{lemma}

	\begin{proof}
        As $G$ has infinitely many ends, it has a subgroup $H$ isomorphic to the free group of rank $2$ (see for example \cite[Corollary 1.3]{AMO07}). Hence, by \cref{th:piantadosi}, $H$ admits a        
		weakly aperiodic SFT $X$, and by \cref{th:jeandel}, $\Xlift$ a weakly aperiodic SFT of $G$. Moreover, note that \cref{th:cohen} implies that $\Xlift$ cannot be strongly aperiodic.
	\end{proof}

	The following lemma is inspired by what is called the Magnus-Moldavansky rewriting process, a technique to write one-relator groups as a tower of HNN extensions on one-relator groups. Part of that technique implies the ``rewriting'' of a one-relator group as another one that satisfies the condition of \cref{lemma:zero-aperiodic}, that is: some generator appears in the relator with total exponent $0$.
	We cut out from this process the groups with infinitely many ends, as they already fall within the bounds of \cref{lemma:infiniteends-aperiodic}.

	\begin{lemma}
		\label{lemma:moldavansky}
		Let $G = \langle S | r \rangle$ with $|S|\geq 3$ 
		and $r$ cyclically reduced. 
		
		Then either $G$ has infinitely many ends, or there exist some presentation $\langle T \mid r' \rangle$ of $G$ and some element $t\in T$ such that:
		\begin{itemize}
			\item $|T| = |S|$
			\item $r'$ is cyclically reduced
			\item $t$ occurs in $r'$
			\item $|r'|_t = 0$
			\item $G \cong  \langle T \mid r'\rangle$
		\end{itemize} 
	\end{lemma}

	\begin{proof}
        First, assume that some element $s$ of $S$ does not occur in $r$, then we write $G$ as $G^\prime \ast \langle s \rangle$ with $G^\prime = \langle S\backslash \{s\} | r \rangle$. 
        Then, as $|S\backslash \{s\}|\geq 2$, and $G^\prime$ has one relator, we have $G^\prime \neq \{1\}$, implying that $G$ has infinitely many ends (see for example \cite[Theorem 1.8]{freeproduct}).
		
		Now, suppose all elements of $S$ occur at least once in $r$. If some element $s\in S$ satisfies $|r|_s=0$, then we conclude by choosing $T:=S$, so we may assume that there is no such element. Up to changing the generating set by replacing elements of $S$ with their formal inverse, we also assume that for each $s\in S$, $|r|_s>0$, and write $S = \{ s_1, \dots, s_n \}$ so that
		\[0 < |r|_{s_1} \leq |r|_{s_2} \leq \dots \leq |r|_{s_n}.\]

		We set $t_1 := s_1s_n$ and $t_i := s_i$ for every $i > 1$. We now let $r'$ be the word obtained after rewriting $r$ using the family $T:=\sg{t_1,\ldots, t_n}$, i.e., after replacing each occurence of $s_i$ (respectively of ${s_i}^{-1}$) in $r$ with $t_i$ (respectively with ${t_i}^{-1}$) for $i>1$, and after replacing each occurence of $s_1$ (respectively of ${s_1}^{-1}$) with $t_1{t_n}^{-1}$ (respectively with $t_n{t_1}^{-1}$). 
		Note that $\langle T \mid r' \rangle$ is obtained from 
		$\langle S\mid r\rangle$ after applying the following Tietze transformations:
		
		\begin{itemize}
		 \item[1] Add new generators $t_1, \ldots, t_n$ together with the relations $t_1 {s_n}^{-1} {s_1}^{-1}$ and $t_i {s_i}^{-1}$ for each $i>1$;
		 \item[2] Remove the letter $s_1$ and replace its occurences with $t_1 {s_n}^{-1}$;
		 \item[3] Remove the letter $s_i$ for each $i>1$ and replace its occurences with $t_i$;
		 \item[4] Remove the redundant relations $t_1 {t_n}^{-1} t_n {t_1}^{-1}$ and $t_i {t_i}^{-1}$ for each $i>1$.
		\end{itemize}

		Hence $\langle T \mid r' \rangle$ is also a presentation of $G$. Moreover, note that

		\[
		0 < |r'|_{t_n} = |r|_{s_n} - |r|_{s_1} < |r|_{s_n},
		\]

		while for any $i>1$, $|r'|_{t_i} = |r|_{s_i}$. If the cyclically reduced form of $r'$ does not contain all $t_i$'s, then $G$ has infinitely many ends. 
		Iterating this process a finite number of times, we either prove $G$ has infinitely many ends, or that we obtain a presentation $\langle T \mid r'\rangle$ such that some $r'$ is a reduced nonempty word and that some element $t\in T$ occurring in $r'$ satisfies $|r'|_t=0$. In particular, $\langle T \mid r'\rangle$ satisfies the desired properties.
	\end{proof}

	We now have all that we need to prove \cref{thm: one-relator}.

	\begin{proof}[Proof of \cref{thm: one-relator}]
		Let $G$ be a one-relator group with at least three generators. Then by \cref{lemma:moldavansky}, either $G$ has an infinite number of ends, and consequently it has a weakly but not strongly aperiodic SFT by \cref{lemma:infiniteends-aperiodic}; or its presentation can be rewritten using \cref{lemma:moldavansky}, and we can apply  \cref{lemma:zero-aperiodic}, which yields the same result.
	\end{proof}

\section{Quasi-planar groups}
\label{sec:quasi-planar}

	In this last section, we give a proof of \cref{cor: QP}.
	It mainly follows from a combination of \cref{lemma:heredity} and \cref{thm: one-relator} together with the following result of MacManus, which characterizes quasi-planar groups.

	\begin{theorem}[\cite{Mac23}, Corollary D]
	\label{th:macmanus}
	A finitely generated group $G$ is quasi-planar, that is it admits a Cayley graph quasi-isometric to a planar graph, if and only if $G$ is virtually a free product of finitely many free groups and surface groups.
	\end{theorem}
	
	\QP*

 \begin{proof}
    First, it is known that virtually cyclic and torsion-free virtually $\Z^2$ groups are periodically rigid (see notably \cite[Prop. 6.4]{Bitar24} for the inheritance of period-ridigity from $\Z^2$ to any torsion-free virtually $\Z^2$ group). We will prove that the opposite implication holds on quasi-planar groups.
 
     Let $G$ be a quasi-planar group. By Theorem \ref{th:macmanus}, $G$ is virtually a free product of a finite number of groups $G_1, \ldots, G_m$, where for each $i\in [m]$, $G_i$ is either a free group or a surface group. If all $G_i$'s are free groups, then $G$ is a virtually free group, and it is thus either virtually cyclic, or it has an infinite number of ends and admits $\mathbb F_2$ as a subgroup, in which case Theorems \ref{th:piantadosi}, \ref{th:cohen} and \ref{th:jeandel} imply that $G$ is not periodically rigid.  
      
      Assume now without loss of generality that $G_1$ is a surface group. If $m=1$, then either $G_1=\mathbb Z^2$ and $G$ is periodically rigid; or $G_1$ is a surface group of genus at least $2$, in which case Theorem \ref{thm: one-relator} implies that it is not periodically rigid, and then by Lemma \ref{lemma:heredity} $G$ is also not periodically rigid. If $m\geq 2$, then by \cite[Proposition 3]{Serre}, $G$ contains $\mathbb F_2$ as a subgroup. In particular, Theorems \ref{th:piantadosi} and \ref{th:jeandel} imply that $G$ admits a weakly aperiodic SFT $X$. We claim moreover that $G$ has an infinite number of ends (see for example \cite[Theorem 1.8]{freeproduct}), hence Theorem \ref{th:cohen} implies that $X$ is not strongly aperiodic, so $G$ is not periodically rigid. 
  \end{proof}

\section{Further questions and remarks}

The first natural question left open by this paper is Conjecture  \ref{conj: Bitar}. The next step would be to extend Theorem \ref{thm: one-relator} to all one-relator groups. As it seems, two-generator one-relator groups form a specific subclass of one-relator groups, on which Conjecture \ref{conj: strongly} remains open. We make here a few observations about special cases for which we can still give a positive answer to this question.

Recall that Conjecture \ref{conj: Bitar} was proved on Baumslag-Solitar groups $\mathrm{BS}(m,n)$ \cite{ME22}.
A finitely generated group is \emph{free-by-cyclic} if it can be expressed as a semidirect product of the form $\mathbb{F}_k\rtimes \Z$ for some $k\in \mathbb N$. Observe that the tools we used here still apply to show that (virtually) free-by-cyclic groups satisfy Conjecture \ref{conj: Bitar}. Let $G$ be a free-by-cyclic group with a decomposition $\mathbb{F}_k\rtimes \Z$ for some $k\in \mathbb N$. If $k=0$, $G$ is cyclic, and if $k=1$, $G$ is either isomorphic to $\mathbb Z^2$, or to the fundamental group $\langle a,b\mid abab^{-1}\rangle$ of the Klein bottle, which also corresponds to the Baumslag Solitar group $\mathrm{BS}(1,-1)$, and is a torsion-free group containing a subgroup of index $2$ isomorphic to $\Z^2$. Now, let us assume $k\geq 2$, and consider the weakly aperiodic SFT $X$ on $\mathbb F_k$ given by Theorem \ref{th:piantadosi}. By Proposition \ref{th:jeandel}, $\Xlift$ is a weakly aperiodic SFT on $G$. We claim that moreover, $\Xlift$ is not strongly aperiodic. To see this, we let $H$ denote a normal subgroup of $G$ isomorphic to $\mathbb F_k$ such that $G/H\simeq \Z$. In particular, there exists some element $a\in G-H$ such that the cosets of $H$ correspond exactly to the sets $a^iH$, for $i\in \Z$. 
For each $x\in X$, we let $x^{\uparrow}$ be the configuration of $G$ defined as follows. For every $g\in G$, we let $i\in \Z, h\in \mathbb F_k$ be the unique elements such that $g=a^ih$, and set $x^{\uparrow}(g):=x(h)$. Note that for every $x\in X$, we must have $x^{\uparrow}\in \Xlift$. Moreover, observe that we also have $a\cdot x^{\uparrow}=x^{\uparrow}$, implying that $\Xlift$ is not a strongly aperiodic SFT on $G$. Combining all of these observations with Proposition \ref{prop: finite-index}, we conclude that every virtually free-by-cyclic group satisfies Conjecture \ref{conj: Bitar}. 

Virtually free-by-cyclic groups occupy an important place in the class of one-relator groups. Kielak and Linton \cite{Kielak_Linton} proved that all finitely generated one-relator group with torsion are virtually free-by-cyclic, and, based on a result of Brown \cite{Brown}, Dunfield and Thurston \cite{Dunfield_Thurston} proved that some positive fraction of the groups admitting a presentation with two generators and one relator are virtually free-by-cyclic.

Eventually, observe that given some finitely generated groups $G,H$ such that $H$ is a subgroup of finite index in $G$, the constructions $\Xlift$ and $\Xrightlift$ behave quite differently in general. 
It is not even clear for us if the two constructions are comparable to each other when $H$ is a normal subgroup of $G$  (e.g.  does there exist morphisms between these subshifts?).

\section*{Acknowledgements}
	The authors would like to thank Nicolás Bitar for mentioning~\cite{Cohen-GS-Rieck}, itself mentioning weak-but-not-strong aperiodicity of hyperbolic groups being obtained through results by Coornaert and  Papadopoulos~\cite{Coornaert-Papadopoulos}. The first author would like to thank Nathalie Aubrun for her help in understanding one-relator groups, and particularly the Magnus-Moldavansky rewriting. The second author would like to thank Louis Esperet for his comments on a previous version of the paper, which was included in his PhD thesis manuscript.

 \bibliographystyle{abbrv}
 \bibliography{biblio}

\end{document}